\documentclass[leqno, 11pt]{amsart}
\usepackage{amsmath,amsfonts,amsthm,mathrsfs,amssymb}
\usepackage{enumerate}
\usepackage{graphicx}
\usepackage{mathtools}
\usepackage{bm}

\newcommand{\limnk}{\underset{\underset{k}{\longrightarrow}}{\lim}}
\newcommand{\limnki}{\underset{t \rightarrow \infty} {\lim}}

\newcommand{\mint}{\underset{t}{\min}}

\theoremstyle{plain}
\newtheorem{thm}{Theorem}[section]

\newtheorem{lem}[thm]{Lemma}

\newtheorem{prop}[thm]{Proposition}

\theoremstyle{definition}

\newtheorem{exmp}[thm]{Example}

\newcommand{\beqn}{\begin{eqnarray*}}
\newcommand{\eeqn}{\end{eqnarray*}}
\newcommand{\beq}{\begin{eqnarray}}
\newcommand{\eeq}{\end{eqnarray}}

\newcommand\lto{\longrightarrow}

\newcommand\ip{\mathfrak{p}}
\newcommand{\p}{{\mathfrak p}}
\newcommand{\q}{{\mathfrak q}}
\newcommand\im{\mathfrak{m}}
\def\aa{\mathfrak{a}}
\def\aad{\mathfrak{ad}}
\def \rr{\mathfrak{r}}
\def\dd{\mathfrak{d}}

\newcommand{\kk}{\Bbbk}

\def\Z{{\mathcal Z}}
\def\R{{\mathcal R}}
\def\Gr{{\mathcal G}}
\def\Fi{{\mathcal F}}

\def\CM{{\mathcal M}}
\newcommand{\m}{{\mathfrak m}}
\def\iP{{\mathbb P}}

\def\C.{C_\bullet}
\def\F.{F_\bullet}
\def\k.{\mathcal{K}_{\bullet}}

\newcommand\ra{\rightarrow}

\newcommand\depth{\textnormal{depth}}

\def\height{\operatorname{ht}}
\def\max{\operatorname{max}}

\newcommand\sym{\textnormal{Sym}}

\newcommand\reg{\textnormal{reg}}

\newcommand{\been}{\begin{enumerate}}
\newcommand{\eeen}{\end{enumerate}}

\usepackage[margin=1in]{geometry}

\usepackage{paralist}

\makeatletter
\usepackage[colorlinks,pagebackref=true]{hyperref}
\@addtoreset{equation}{section}
\def\theequation{\thesection.\@arabic \c@equation}

\def\@citecolor{black}
\def\@linkcolor{black}
\def\@urlcolor{black}

\def\newblock{\hskip .11em plus .33em minus .07em}

\title{Curves in $\iP^n$ of analytic spread at most $n$}

\author{Marc Chardin}
\address{
Institut de Math{\'e}matiques de Jussieu\\
CNRS \& Sorbonne Universit{\'e}\\
4, place Jussieu\\
F--75252 Paris cedex 05\\
France}
\email{marc.chardin@imj-prg.fr}
\author{Clare D'Cruz}
\address{Chennai Mathematical Institute, Plot H1, SIPCOT IT Park,   Chennai 603 103, India}
\email{clare@cmi.ac.in}
\begin{document}

\maketitle

\tableofcontents

 \begin{abstract}
We study   closed subschemes $X$ in  $\iP^n$ of dimension one, locally defined  at any point by at most $n$ equations such that  the analytic spread  of  $I_{\m}$  is at most $n$, where $I  \subseteq  \kk[x_0, \ldots, x_n] $ is the defining ideal of $X$ and $\m = (x_0, \ldots, x_n)$. In this situation, we show that, under mild conditions, all the powers  of $I_{\m}$ have positive depth, hence the limit depth of $I_{\m}$ is $1$ unless $I$ is a complete intersection. Moreover, the regularity of the Rees ring is at most one and the fiber cone is Cohen-Macaulay.   This applies  to every ideal defining a  monomial curve in $\iP^3$.
 \end{abstract}

\section{Introduction.}\label{sec:intro}
Let  $(R, \m)$ be a local ring  of  finite dimension with infinite residue field and $I$ an ideal of positive height in $R$. The analytic deviation, the depth function  and the Castelnuovo-Mumford regularity of blowup algebras,  namely
 the Rees algebra  $\R_I := \oplus_{t \geq 0} I^t$, the associated graded ring  
$\Gr_I := \R_I / I \R_I$   and the fiber cone  $\Fi_I := \R_I \otimes R/ \m$  have been investigated in the past few decades. 
  The analytic spread of $I$, $\aa(I) := \dim (\Fi_I)$ has  played an important role in the study of blowup algebras.  It is well known that  $\height(I) \leq \aa(I)$,   where 
  $\height(I)$ denotes the height of $I$ \cite{n-rees-1954}.  Motivated by this result, Huneke and Huckaba defined the analytic deviation of $I$ as  $\aad(I):= \aa(I)  - \height(I)$  \cite{huneke-huckaba-powers}. Ideals of small analytic deviation were particularly investigated. One of the remarkable results was by  Cowsik and Nori and they   showed that if $\p$ is a prime ideal in a Cohen-Macaulay domain $R$ such that $R_{\p}$ is regular and $\dim(R/ \p)=1$, then $\p$ is a complete intersection if and and only if  $\aad(\p)=0$ \cite[Proposition~3]{cowsik-nori}. 
Motivated by this result,   Huneke and Huckaba studied ideals $I$ of analytic deviation at most two  \cite{huneke-huckaba-powers}.     
For an ideal $I$ in a Cohen-Macaulay local ring $R$  of analytic deviation one, under some mild assumptions  on $I$, they give necessary and sufficient conditions for $\Gr_I$  \cite[Theorem~2.9]{huneke-huckaba-powers} and  $\R_I$ to be Cohen-Macaulay \cite[Thereom~2.1]{huneke-rees}.
 In \cite[Therorem~1]{zarzuela}, Zarzuela  obtained interesting results on the  depth  of the Rees algebra $\R_I$ and the associated graded ring $\Gr_I$ for  generically complete intersection ideals with analytic deviation one. Properties of the blowup algebras $\R_I$ and $\Gr_I$ for ideals of analytic deviation one has also been studied in \cite{goto-huckaba}, \cite{goto-nakamura-nishida}, \cite{goto-nakamura}, \cite{trung0} and \cite{vascon}.
 
Interest in the depth function  $\dd_I(t):=\depth(R/ I^t) $  goes back to the work of Burch  which gives the following inequality for any ideal $I$ in a Noetherian  local ring $(R, \m)$:
 ${\displaystyle \mint~\dd_I(t)\leq \dim(R)- \aa(I)}$ \cite{burch}.  In \cite{brodmann}, Brodmann showed that  the asymptotic value of $\dd_I(t)$ exists and used his result to improve on Burch's inequality. In \cite{herzog-hibi},  Herzog and Hibi  studied the asymptotic nature of $\dd_I(t)$ for several ideals which arise from combinatorics. They
called the asymptotic value of $\dd_I(t)  $ the  {\em limit depth of} $I$  which they denoted by $\limnki~   \dd_I(t) $.  
The depth function has also been studied in \cite{bandari}, \cite{herzog-sqfree},  \cite{mastuda} and \cite{zarzuela}. 

The Castelnuovo-Mumford regularity of blowup algebras has been of interest in the last few decades.
In \cite[Proposition~4.1]{JU}, Johnson and Ulrich showed that for any proper ideal
 in  a Noetherian local ring $R$,  the regularity of $\R_I$ and $\Gr_I$  are equal.  
 Equality of the  Castelnuovo-Mumford regularity of  $\R_I$ and $\Gr_I$ has also been studied in  \cite{trung}, 

In this paper we consider   curves in $\iP^n$ whose defining ideal $I \subseteq S= \kk[x_0, \ldots, x_n]$  satisfy the condition  $\aad(I_{\m}) \leq 1$. 
 We  first identify conditions for an  ideal $I$ in a Noetherian unmixed local ring so that the reduction number is at most one, the depth is at least two and the 
 regularity of $\R_I$ is at most one
(see Lemma~\ref{RegRees}).  We apply these results  to obtain our main result in section 2 (see Theorem~\ref{main-thm}).
 We  
show that if $I$ is the defining ideal of a closed subscheme $X$, then under some assumptions, if  $\aad(I) \leq 1$, 
then  $\dd_{I_{\m}}(t)=1$ and  hence $\limnki \dd_{I_{\m}}(t)=1$ (see Theorem~\ref{main-thm}). This gives a   precise value  to the limit depth of $I$.  Moreover, $\reg(R_I)$ is at most one and we conclude that the fiber cone $\Fi_I$ is 
Cohen-Macaulay. This gives an estimate on the number of generators of $I^t$ for all $t \geq 1$. 
As a consequence we show that if $I$ is the defining ideal of a 
monomial curve in ${\iP}^3$  then all the conditions of Theorem~\ref{main-thm} are satisfied (see Example~\ref{example-p3}). 

In  section 3 we consider examples of monomial curves in  ${\iP}^4$. In Example~\ref{ex1} we give an infinite class of 
examples of curves whose defining ideal is generated by a $d$-sequence. In Example~\ref{ex2} and Example~\ref{ex3} we show that there exists an 
infinite class of  monomial curves  in ${\iP}^4$  with analytic spread four but not all the conditions of Theorem~\ref{main-thm} are 
satisfied.

{\bf Acknowledgement:}
 The second author thanks `The Foundation Sciences Math{\'e}matiques de Paris' for financial support,  Institut de Math{\'e}matiques de Jussieu,
CNRS \& Sorbonne Universit{\'e}  for local hospitality, partial support from INFOSYS and 
MATRICS grant MTR/2023/000661 funded by ANRF, Government of India. The computer software Macaulay2 \cite{M2} was a crucial input in  our investigations.

\section{Main result}
  In this section,  under some assumptions we show that the blowup algebras of 
 certain closed subschemes of dimension one in $\iP^n$  have a  nice behaviour.

Let $(R, \m)$ be a commutative ring and let  $S = \oplus_{n\geq 0} S_n$ be a finitely generated standard graded ring over $S_0 = R$. Let $S_{+}$ denote the the ideal generated by the homogeneous elements of positive degree of $S$ and let  $M$ be a graded $S$ module. 
Set $a(M) := \sup\{n | M_n \not = 0\}$. 
For $i \geq 0$, let $a_i(S) := a(H^i_{S+}(S))$, where $H^i_{S+}(S)$  is the $i$-th local cohomology functor with respect to $S_{+}$. The Castelnuovo-Mumford regularity of $S$ is defined as $\reg(S) := \max\{ a_i(S) + i: i \geq 0\}$.

We  recall results concerning the approximation complexes defined by Herzog, Simis and Vasconcelos \cite{HSV}. These are defined from a finite set of generators of an ideal and their homology is independent of the choice of generators, up to isomorphism. They provide information on $\R_I$ and $\Gr_I$, from information on Koszul cycles or Koszul homology built from the given generators.

In the following proposition, concerning the acyclicity of these complexes, a complex $C_\bullet$ {\it resolves} a module $M$ if $H_0(C_\bullet )$ is isomorphic to $M$ and $H_i(C_\bullet )=0$ for $i\not= 0$ (equivalently there is a map from $C_0$ to $M$ which is a quasi-isomorphism from $C_\bullet$ to the complex reduced to $M$ in homological degree 0).

\begin{prop}\label{RegZero}
Let $R$ be a Noetherian local ring and $I$ an $R$-ideal. The following are equivalent~:

(i) The $\Z$-complex resolves $\R_I$,

(ii) The $\CM$-complex resolves $\Gr_I$,

(iii) The $\CM$-complex has only $0$-th homology,

(iv) $\reg (\R_I )=0$,

(v) $\reg (\Gr_I  )=0$.

If further $R$ has infinite residue field, these are also equivalent to :

(vi) $I$ is generated by a $d$-sequence.
\end{prop}

Recall that  $I\R_I =(\R_I)_+ (-1)\subset \R_I$ and $\R_I/I\R_I =R/I\oplus_R I/I^2 \oplus_R \cdots$. 

\begin{proof} Replacing $(R, \im ,k)$ by $(R(U),\im R(U), k(U))$ (where $U$ is an indeterminate), as in \cite[p.~17]{nagata}, properties (i) to (v) are unchanged. 
Hence we may assume that the residue field is infinite. 

By \cite[4.1]{JU}, (iv) and (v) are equivalent. By \cite[12.9]{HSV}, (i), (iii) and (vi) are 
equivalent. The implication (ii)$\Rightarrow$(iii) is trivial and the reverse one 
holds since (i) is implies by (iii) and shows that $\sym_R (I/I^2)=\R_I /I\R_I$
by  \cite[3.1]{HSV}. Finally (v) and (vi) are equivalent by  \cite[12.7, 12.8 and 12.10]{HSV}.
\end{proof}

\begin{lem}
\label{RegRees}
Let $(R,\im )$ be a Noetherian unmixed local ring and $J \subset \im$ be an ideal generated by $r+1$ elements such that $\depth~J\geq r \geq 1$. 
Assume that $\dim R\geq r+1$,
 $R$ satisfies  $S_{r+1}$
and
 $J_\ip$ is of linear type for all $\ip \supset J$  with $\height(\ip ) = \height(J)$. Then,
 
(a) $J$ is of linear type and $\reg (\R_J)=0$.\smallskip

(b)  Let $I:=J^{sat}=\cup_i (J:\im^i)$. If $\depth ~R \geq r+2$, then $H^0_\im (\sym_R (J))=0$,
$$
H^1_\im (\sym_R (J))\simeq (I/J)\otimes_R S\{ -1\} ,
$$
$IJ=I^2$, $\depth_\im ( I^t)\geq 2$ for any $t$ and $\reg (\R_I)=1$ unless $I=J$.
\end{lem}
\begin{proof} Notice that if $\depth~J>r$, then $J$ is
a complete intersection ideal, hence (a) and (b) follows. 
Else, $\depth~J=r$ and hence $J$ has height $r$ or $r+1$. 
Since $\depth~J \geq r$,  $H_i=0$ for $i\geq 2$.
 Consider the approximation complex on the $r+1$ generators of $J$ :
$$
\Z_{\bullet} \hphantom{space} 0\ra Z_r\otimes_R S\{ -r\} \ra \cdots \ra Z_1\otimes_R S\{ -1\} \ra Z_0 \otimes_R S\ra 0
$$
with $S:=R[T_1,\ldots ,T_{r+1}]$ and
$$
0\ra H_1\otimes_R S\{ -1\} \ra H_0 \otimes_R S\ra 0
$$
whose 0-th homology are respectively $\sym_R(J)$ and $\sym_R(J/J^2)$. 

Recall  that the first of these is acyclic, as $J$ is
an almost complete intersection. 

Let $\ip$ be a minimal prime of $J$. If $\height (\ip )=r$, then
$J_{\ip}$ is of linear type, hence by Proposition \ref{RegZero} 
$(i)\Rightarrow (iii)$, the
second complex is acyclic locally at $\ip$. If $\height (\ip )=r+1$, 
then, locally at $\ip$, $R_{\ip}$ is Cohen-Macaulay and therefore
$J_{\ip}$ is a complete intersection.

For (a), according to Proposition \ref{RegZero}, it suffices to prove that the depth of
$H_{1}$ is positive locally at any prime $\ip \supset J$ which
is not minimal, since this implies that the $\CM$-complex has only $0$-th homology. Such a  prime has height at least $r+1$.
Hence, as $R$ satisfies $S_{r+1}$, the depth of $B_{1}$ and
of $Z_{1}$ are at least $2$ locally at $\ip$, thus indeed the depth of $H_1$ is positive. 

Notice that $\depth_{\im} (B_i)=\depth_\im (R)-r+i$, for
all $i>0$, hence $\depth_{\im} (Z_i)=\depth_\im (R)-r+i$, for
all $i\geq 2$.

For (b), we compare the two spectral sequences coming from the double complex ${\mathcal C}^{\bullet}_{\im}\Z_{\bullet}$. The acyclicity of $\Z_\bullet$ implies that the total homology is the local cohomology with support in $\im$ of $H_0 (\Z_\bullet )=\sym_R (J)$. On the other hand, the spectral sequence with first terms $^1{E}^p_q = H^p_\im (\Z_q )=H^p_\im (Z_q )\otimes_R S\{ -q\}$ shows that :
$$
 H^1_\im (\sym_R (J))\simeq {^\infty{E}}^2_1 = {^1{E}}^2_1 =H^2_\im (Z_{1})\otimes_R S\{ -1\}\simeq H^1_\im (J)\otimes_R S\{ -1\}\simeq(I/J)\otimes_R S\{ -1\}.
 $$
Indeed, ${^1{E}}^p_p =0$ for all $p$ and ${^1{E}}^{p+1}_p =0$ for $p\neq 1$ since $\depth_\im (Z_i)\geq i+2$ for
$i\geq 2$ ; as ${^1{E}}^2_0 =0$ since $\depth_\im (Z_0)\geq r+2\geq 3$ it follows that  ${^\infty{E}}^2_1 = {^1{E}}^2_1$ and ${^\infty{E}}^{p+1}_p =0$ for $p\neq 1$. This shows that $H^1_\im (\sym_R (J))\simeq {^\infty{E}}^2_1 = {^1{E}}^2_1$ ; furthermore $H^0_\im (\sym_R (J))=0$ since ${^1{E}}^p_p =0$ for all $p$.

 As $\sym_R (J)=\R_{J}$ by (a), it follows that $(I/J)\otimes_R S_{t-1}\simeq (J^{t})^{sat}/J^{t}$
 via the map sending the class of $a\otimes T^{\alpha}$ to $a\phi (T^{\alpha})$ were $\phi :S\rightarrow \R_{J}$ is the map defined by the given generators of $J$. Hence $(I^t)^{sat}=(J^t)^{sat}=IJ^{t-1}\subseteq I^t$ for any $t$, showing that $I^t$ is saturated and equal to $IJ^{t-1}$. 
 The exact sequence 
 $$
 0\ra \R_J \ra \oplus_t (J^t)^{sat} \ra H^1_\im (\R_J)\ra 0
 $$
 provides the exact sequence:
  $$
 0\ra \R_J \ra \R_I \ra (I/J)\otimes_R S\{ -1\}\ra 0,
 $$
 proving that $\reg (\R_I )\leq 1$. Equality holds, unless $I=J$, since writing $S_+ =(T_1,\ldots ,T_{r+1})$, one has an
 exact sequence 
 $$
 H^{r+1}_{S_+}(\R_J )\ra  H^{r+1}_{S_+}(\R_I ) \ra  (I/J)\otimes_R H^{r+1}_{S_+}(S)\{ -1\} \ra 0
 $$
 and $H^{r+1}_{S_+}(\R_J )_{-r}=0$, because  $\reg (\R_J )=0$, showing that  
 $$H^{r+1}_{S_+}(\R_I )_{-r}=(I/J)\otimes_R H^{r+1}_{S_+}(S)_{-r-1}=I/J.$$
   \end{proof}

We now prove the main result in our paper. Let $I$ be an ideal in a local ring $(R, \m)$. 
 Let $\mu(I)$ denote the minimal number of generators of $I$.
 An ideal $J \subseteq  I$ is said to be a reduction of $I$ if   $JI^n = I^{n+1}$ for some $n \geq 0$. A reduction $J$ is said to be a minimal reduction of $I$ if $J$ is minimal  with respect to  inclusion among all the reductions of $I$. If $J$ is a reduction of $I$, then the reduction number of $I$ with respect to $J$ is defined to be $\rr_J(I) = \min\{n \geq 0 | JI^n = I^{n+1}\}$.
The reduction number of $I$ is defined to be
$\rr(I) = \min\{r_J(I)| J { \mbox{ a reduction of }} I \}$. 

\begin{thm}
\label{main-thm}
Let $X$ be a closed subscheme of ${\iP}^n$ of dimension 1, locally defined at any point by at most $n$ equations. 
Let $I\subset k[X_0,\ldots ,X_n]$ be the defining ideal of $X$ and set $\im :=(X_0,\ldots ,X_n)$. Assume that $I_\ip$ is a complete intersection for any minimal prime of $I$. If the analytic spread $\aa (I_\im)$ in $R_\im$ is at most $n$, then :

\been
\item $\depth (R/I^t)>0$ for all $t$.

\item $\reg (\R_I)\leq 1$. In particular, the reduction number  
$\rr(I_\im)$ is at most 1  and the Rees algebra is defined by linear and quadratic equations.

\item
$\Fi_{I}$ is Cohen-Macaulay.

\item
Put $a= \aa (I_\im)$.  Then
$I^t$ is minimally generated by
$ {t + a-1 \choose a-1}
+(\mu(I)-a){t + a-2 \choose a-1}
$
elements.
\eeen
\end{thm}

\begin{proof}
The statements are local at $\im$, so we may localize $R$ and $I$ at $\im$. We may also assume that $k$ is infinite. 
Let $\ip_i$ for $i=1,\ldots ,m$ be the finite number of primes of dimension 1 where $I_{\ip_i} \subset R_{\ip_i}$ is not a complete intersection - these are all among the associated primes of dimension $1$. 
By hypothesis, $I_{\ip_i}\subseteq R_{\ip_i}$ is generated by $n$ elements. Hence $n$ general $k$-linear combinations of the generators of  $I_\im$
generate an ideal $J$ such that $J_{\ip_i}=I_{\ip_i}$ for all $i$; we may further assume that these $n$ general elements generate a
reduction of $I_{\im}$ \cite[Theorem~1, page 150]{n-rees-1954}. In this case, as $J$ is a reduction of $I_\im$, $J$ and $I_\im$ coincide
locally at any prime where $I_\im$ is a complete intersection.  It follows that $I_\im = J^{sat}$, since $\im$ is not associated to $I_{\im}$ and $J$ and $I_\im$ coincide on the punctured spectrum. 

Now (1) and (2) follows from Lemma \ref{RegRees} applied to $J\subset R_\im$ with $r=\mbox{analytic spread of } I_{\im}$.

(3) As the reduction number of $I_{\im}$ is at most $1$, and the
analytic deviation is at most $1$, $\Fi_{I_{\im}}$ is Cohen-Macaulay (by \cite[Corollary~(1a)]{shah},  \cite[Theorem~4.2]{cort-zar}). Hence $\Fi_{I}$ is Cohen-Macaulay, since $I$ is homogeneous.

(4) As   $\Fi_{I_{\m}}$  is Cohen-Macaulay (\cite[Theorem~5.7]{cort-zar}) and $I$ is a homogenous ideal, 
$$
\mu (I^t) = \mu (I_{\m}^t)=
 {t + a-1 \choose a-1}
+(\mu(I)-a){t + a-2 \choose a-1}.
$$
\end{proof}

\begin{exmp}
\label{example-p3}
According to a result of Gimenez, Morales and Simis \cite{GMS}, the analytic spread of the ideal of a monomial curve in ${\iP}^3$
has analytic spread at most $3$, in the localization of $R:=k[X_0,\ldots ,X_3]$ at its unique homogeneous maximal ideal. Hence, by the above result, if $I\subset R$ is its defining ideal :

(1) $\depth (R/I^t)>0$ for all $t$,

(2) $\R_I$ has regularity 1 unless $I$ is generated by at most $3$ elements in which case $\reg (\R_I) =0$. The latter condition holds if and
only if $I$ is perfect.
 
(3) $\Fi_{I}$ is Cohen-Macaulay and for all $t \geq 1$,
$\mu(I^t) = (\mu  (I) -2 ){t+2 \choose 2} - (\mu(I)-3)  (t+1)$. 
\end{exmp}

\section{Examples in $\iP^4$}
\label{section:examples}

The monomial curve parametrized by
$x_{0}=1, x_{1}=w^{a_{1}}, x_{2}=w^{a_{2}} , x_{3}=w^{a_{3}}$ and  $x_{4}=w^{a_{4}}$ on the affine chart
$X_{0}=1$ of ${\iP}^{4}_{k}$ will be called ``the monomial curve of
degrees $(a_{1},a_2, a_3, a_{4})$''. We will denote this curve by ${\mathcal C}(a_1,a_2,a_3,a_4)$.
In this section we consider  monomial curves  ${\mathcal C}(1,2,3,3a+i)$ where  $i=0,1,2$. We  show that if $i=0$, then all the assumptions of Theorem~\ref{main-thm} are satisfied and  if $i=1,2$, then at least one of the assumptions of Theorem~\ref{main-thm} is not satisfied. 
\begin{exmp}
\label{ex1}
Let $a \geq 3$ and let ${\p}_a
\subseteq R := k[x_0, \ldots, x_4]$ be the defining ideal of the monomial curve 
 ${\mathcal C}(1,2,3,3a)$ and $\m = (x_0, \ldots, x_4)$. 
 Then
 \been
 \item
 \label{ex1-1}
$ \p_a =  (g_1, g_2, g_4, g_3)$ where 
$g_1 = x_2^2-x_1x_3$, $g_2= x_1x_2-x_0x_3$, $g_3 =x_3^a - x_0^{a-1} x_4$ and $g_4 = x_1^2-x_0x_2$.

\item
\label{ex1-2-0}
Let $I= (\p_a)_{\m}$. Then
\been
\item
\label{ex1-2}
$I$ is generated by a $d$-sequence and $\rr(I) =0$,

\item
 \label{ex1-3}
 $\R_I \cong \sym_R(I)$, $ \R_I $ is Cohen-Macaulay and $\reg(\R_I)=0$,

\item
 \label{ex1-5}
$\Fi_I  \cong k[u_1, \ldots, u_4] $ where $u_1, \ldots, u_4$ are variables 
and  for all $t \geq 1$, 
${\displaystyle \mu(\p_a^t) = {t+3 \choose 3}}$.
\eeen
 \eeen
\end{exmp} 
\begin{proof} (\ref{ex1-1}): Put $I_a = (g_1, g_2, g_3, g_4)$. One can verify that 
$I_a \subseteq \p_a$,    $x_0, x_4$ is a system of parameters for $R/ I_a$ and 
$I_a + (x_0, x_4) = (x_0, x_1^2, x_1x_2, x_2^2-x_1 x_3, x_3^a,  x_4)$. Putting the graded-rev-lex order  with $x_0 > \cdots > x_4$, we get that  the initial  ideal of $I_a + (x_0, x_4)$ is $(x_0, x_1^2, x_1x_2, x_2^2, x_3^a,  x_4)$. 
Moreover, $(x_1^2, x_1x_2, x_2^2, x_3^a) = (x_1, x_2^2, x_3^a) \cap (x_1^2, x_2, x_3^a)$. Hence, we have  the exact sequence 
\begin{align*}
  0 
   \longrightarrow \frac{k[x_1, x_2, x_3]}{ (x_1^2, x_1x_2, x_2^2, x_3^a)}
   \longrightarrow \frac{k[x_1, x_2, x_3]}{ (x_1, x_2^2, x_3^a) }
            \oplus \frac{k[x_1, x_2, x_3]}{ (x_1^2, x_2,  x_3^a)}
   \longrightarrow \frac{k[x_1, x_2, x_3]}{ (x_1, x_2, x_3^a)}
   \longrightarrow 0,
   \end{align*}
which implies that 
   \beq
   \label{length 3a} 
   \ell\left( \frac{k[x_1, x_2, x_3]}{ (x_1^2, x_1x_2, x_2^2, x_3^a)}\right)
   &=& \ell\left( \frac{k[x_1, x_2, x_3]}{ (x_1,  x_2^2, x_3^a)} \right)
   + \ell\left(  \frac{k[x_1, x_2, x_3]}{ (x_1^2, x_2,  x_3^a)} \right)
   - \ell \left( \frac{k[x_1, x_2, x_3]}{ (x_1, x_2, x_3^a)}\right)\\ \nonumber
   &=& 2a + 2a -a  \\ \nonumber
   & = & 3a.   
   \eeq
 Put $\m= (x_0, x_1, x_2, x_3, x_4)$. Then
\begin{align}
\label{multiplicity 3a} 
  3a 
& =  e( (x_0, x_4) ; R/ \p_a) \\  \nonumber
&=  e( (x_0, x_4) ; R_{\m}/ \p_a R_{\m})  & \mbox{\cite[Theorem~15, page 333]{northcott-lessons}}\\ \nonumber
&\leq   \ell \left(  \frac{R_{\m} }{ (\p_a + (x_0, x_4)) R_{\m}} \right) & \\ \nonumber
&\leq   \ell \left(  \frac{R}{ I_a + (x_0, x_4)} \right) & \\ \nonumber
&= \ell \left(  \frac{R}{ (x_0, x_1^2, x_1x_2, x_2^2, x_3^a,  x_4)} \right) & \\ \nonumber
&= \ell \left(  \frac{k[x_1, x_2, x_3]}{ (x_1^2, x_1x_2, x_2^2, x_3^a)} \right) & \\ \nonumber
&= 3a & \mbox{[from (\ref{length 3a})]} .
\end{align}
Hence $I_a  + (x_0, x_4) = \p_a + (x_0, x_4)$ which implies that $\p_a  + ( x_0)\subseteq  I_a + (x_0, x_4)$. 
Let $\alpha \in \p_a + (x_0)$.  Then 
$\alpha = \beta  + x_4 r$ where $\beta \in  I_a + (x_0)$  and $r \in R$. This implies that 
$x_4 r = \alpha - \beta \in \p_a + (x_0) $ as $I_a \subseteq \p_a$.  Hence $r \in  \p_a + (x_0) : (x_4)$. By  (\ref{multiplicity 3a}),  $ x_4$ is a regular element on $R/ \p_a + (x_0) $   which implies that   $(\p_a + x_0) : (x_4)   = \p_a + (x_0)$. 
This implies that $r \in  \p_a + (x_0)$ and hence  $\p_a + (x_0) \subseteq I_a  + (x_0) + x_4(p_a + (x_0)) \subseteq I_a  + (x_0) + \m(p_a + (x_0))\subseteq \p_a + (x_0)$. By graded Nakayama lemma, $\p_a + (x_0) = I_a  + (x_0)$.  
This implies that $\p_a \subseteq I_a + (x_0)$. Let $f \in \p_a$. Then $f = g + x_0 h$ where $g \in I_a$ and $h \in R$. 
This implies that $x_0h = f-g \in \p_a$ as $I_a \subseteq \p_a$.  By  (\ref{multiplicity 3a}), as $x_0$ in nonzerodivisor on $R/\p_a$. 
Hence $h \in \p_a : (x_0) = \p_a$. 
This gives $\p_a \subseteq I_a  +  x_0 \p_a  \subseteq I_a  + \m \p_a \subseteq \p_a$. By graded Nakayama lemma, $\p_a = I_a$.  
  This proves (\ref{ex1-1}).

(\ref{ex1-2}): For all $2 \leq i \leq k \leq 4$, $g_1, g_i$ is a regular sequence. Hence 
\beqn
((g_1):g_i g_k) &=& ((g_1) : g_i) : g_k) =  ((g_1): g_k).  
\eeqn
We also have 
\beqn
(g_1, g_2) : (g_3^2) &=&  ( (g_1, g_2) : g_3)= ((g_1, g_2, x_3 g_4)   : g_3) = (g_1, g_2, x_3 g_4),\\
(g_1, g_2) : (g_3g_4) &=&( ((g_1, g_2) : g_3)  : g_4 ) = (g_1, g_2, x_3 g_4) : (g_4) = (x_2, x_3),\\
(g_1, g_2): g_4 &=& (x_2, x_3) =  (g_1, g_2): g_4^2 .
\eeqn
By \cite{huneke1}, $g_1, g_2, g_3, g_4$ is a $d$-sequence and hence they are analytically independent (\cite[Theorem~2.2]{huneke1}). Hence $\rr(I) =0$.

(\ref{ex1-3}):  Since $I$ is generated by a $d$-sequence,   $\R_I \cong \sym_R(I)$ (\cite[Theorem~3.1]{huneke3}) and $\reg(\R_I)=0$   (\cite[Corollary~1.2]{trung}).  

(\ref{ex1-5}): By \cite[Theorem~2.2]{huneke1}, 
$
 \Fi_I  \cong  \R_I /\m  \R_I =  k[x_0, \ldots, x_4, u_1, \ldots, u_4]  / \m 
 \cong  k[u_1, \ldots, u_4] 
$. 
It follows that $\reg(\Fi_I)=0$ and that  
${\displaystyle  \mu(\p_a^t) = \mu(I^t) = {t+3 \choose 3}}$.

\end{proof}

Our next example are curves in $\iP^4$ of analytic spread $4$. However, some of  the assumptions in Theorem~\ref{main-thm} are  not satisfied. 
\begin{exmp}
\label{ex2}
Let $a \geq 3$ and let ${\p}_a
\subseteq R := k[x_0, \ldots, x_4]$ be the defining ideal of the monomial curve 
 ${\mathcal C}(1,2,3,3a+1)$. 
 Then
 \been
 \item
 \label{ex2-1}
$ \p_a =  (g_1, \ldots g_6)$ where 
 $g_1 =  x_2^2-x_1x_3$, 
 $g_2 = x_1x_2-x_0x_3$,
 $g_3 = x_1^2-x_0x_2$,   
 $g_4= x_3^{a+1}-x_0^{a-1}x_2x_4$,   
 $g_5= x_2x_3^a-x_0^{a-1}x_1x_4$ and 
 $g_6 = x_1x_3^a-x_0^ax_4$. 
 
 \item
 \label{ex2-2}
 Put $I = (\p_a)_{\m}$. Then:
 \been
 \item
  \label{ex2-2-1}
$\aa (I)=4$.
   
   \item
    \label{ex2-2-2}
$\rr(I) =2$.

\item
 \label{ex2-2-3}
${\displaystyle \mu(\p_a^t)  = {t+3 \choose 3} + 2 {t+2 \choose 3} + {t+1 \choose 3}.}$
 \eeen
 
  \item
  \label{ex2-4} $\reg ( \R_I) \geq 2$. 
\eeen
\end{exmp}
\begin{proof}
(\ref{ex2-1}): Put  $I_a = (g_1, \ldots, g_6)$. One can verify that $ I_a \subseteq \p_a$.
Note that  $x_0, x_4$ is a system of parameters for $R/ I_a$ and 
$I_a + (x_0, x_4)=  (x_0, x_1^2,  x_1x_2, x_2^2-x_1x_3,   x_3^{a+1},  x_1x_3^a, x_2x_3^a, x_4).$ Putting the 
graded-rev-lex order,  we get that  the leading ideal of $I_a$ is   $(x_0, x_1^2, x_1x_2, x_2^2, x_3^{a+1},  x_1x_3^a,  x_2x_3^a,  x_4)$. 
Moreover, $( x_1^2, x_1x_2, x_2^2, x_3^{a+1},  x_1x_3^a,  x_2x_3^a) 
= ( x_1,  x_2^2, x_3^{a}) \cap ( x_1^2, x_2, x_3^{a})\cap ( x_1, x_2, x_3^{a+1})$. Hence from the exact sequences 
\begin{align}
  \label{length 3a_1-2}
      0 
   \rightarrow \frac{k[x_1, x_2, x_3]}{ ( x_1^2, x_1x_2, x_2^2, x_3^{a+1},  x_1x_3^a,  x_2x_3^a) }
   \rightarrow \frac{k[x_1, x_2, x_3]}{ (x_1^2, x_1x_2, x_2^2, x_3^{a})  }
            \oplus \frac{k[x_1, x_2, x_3]}{ ( x_1, x_2, x_3^{a+1})}
   \rightarrow\frac{k[x_1, x_2, x_3]}{ (x_1, x_2, x_3^{a})}
   \rightarrow 0\\
\label{length 3a_1-1}
   0 
   \rightarrow \frac{k[x_1, x_2, x_3]}{ ( x_1^2, x_1x_2, x_2^2, x_3^{a}) }
   \rightarrow \frac{k[x_1, x_2, x_3]}{ ( x_1,  x_2^2, x_3^{a})  }
            \oplus \frac{k[x_1, x_2, x_3]}{ ( x_1^2, x_2, x_3^{a})}
   \rightarrow\frac{k[x_1, x_2, x_3]}{ (x_1, x_2, x_3^{a})}
   \rightarrow 0
  \end{align}
   we get 
   \begin{align}
   \label{length 3a+1} 
  & \ell\left( \frac{k[x_1, x_2, x_3]}{ (x_1^2, x_1x_2, x_2^2, x_3^{a+1},  x_1x_3^a,  x_2x_3^a)}\right) & \\ \nonumber
  &= \ell\left( \frac{k[x_1, x_2, x_3]}{ (x_1^2, x_1x_2, x_2^2, x_3^{a})} \right) 
  + \ell\left(  \frac{k[x_1, x_2, x_3]}{ (x_1, x_2, x_3^{a+1})} \right)
    - \ell \left( \frac{k[x_1, x_2, x_3]}{ (x_1, x_2, x_3^a)}\right)  
    & \mbox{[from \ref{length 3a_1-2}]} \\ \nonumber
   &= \ell\left( \frac{k[x_1, x_2, x_3]}{ (x_1,  x_2^2, x_3^a)} \right)
   + \ell\left( \frac{k[x_1, x_2, x_3]}{ (x_1^2,  x_2, x_3^a)} \right) 
   -  \ell \left( \frac{k[x_1, x_2, x_3]}{ (x_1, x_2, x_3^a)}\right) & \\ \nonumber
  & + \ell\left(  \frac{k[x_1, x_2, x_3]}{ (x_1, x_2, x_3^{a+1})} \right)
    - \ell \left( \frac{k[x_1, x_2, x_3]}{ (x_1, x_2, x_3^a)}\right) 
    & \mbox{[from \ref{length 3a_1-1}]}\\ \nonumber
   &= 2a + 2a -a  + (a+1) - a&\\ \nonumber
   &= 3a+1. &  
   \end{align}
Put $\m= (x_1, x_2, x_3, x_4, x_5)$. Then
 \begin{align}
\label{multiplicity 3a+1} 
            3a +1
& =       e( (x_0, x_4) ; R/ \p_a) &\\ \nonumber
&=       e( (x_0, x_4) ; R_{\m}/ \p_a R_{\m})  &\mbox{\cite[Theorem~15, page 333]{northcott-lessons}} \\ \nonumber
&\leq   \ell \left(  \frac{R_m }{ (\p_a + (x_0, x_4)) R_{\m}} \right) & \\ \nonumber
&\leq   \ell \left(  \frac{R}{I_a + (x_0, x_4)} \right) & \\ \nonumber
&=      \ell \left(  \frac{k[x_1, x_2, x_3]}{( x_1^2, x_1x_2, x_2^2, x_3^{a+1},  x_1x_3^a,  x_2x_3^a)} \right) & \\ \nonumber
&=      3a+1 &  \mbox{[from (\ref{length 3a+1})]} .
\end{align}
Hence $I_a  + (x_0, x_4) = \p_a + (x_0, x_4)$.
The rest of the proof is similar to the proof in Example~\ref{ex1}(\ref{ex1-1}).
This proves (\ref{ex2-1}). 

\noindent
(\ref{ex2-2})
Let $\q = (g_1, g_2, g_3+g_4, g_6)S$ and $J = \q R$. 
We claim that  $J $ is a minimal reduction of $I=(g_1, g_2, g_3+g_4, g_6, g_3, g_5) $.  We have the following relations:
\beq
\label{ex2-minred-g4}
g_3^2 &=& g_1g_6-g_2 g_5+g_3(g_3+g_4)\\ 
\label{ex2-minred-g5}
        g_5^2 &=& -x_0^{a-2}x_3^{a-1}x_4 g_1(g_4+g_3)
 + (x_0^{a-2}x_3^{a-1}x_4+x_3^{a-1})g_1g_4 
+ x_0^{a-2}x_3^{a-1}x_4g_2^2 \\ \nonumber&&
 -x_0^{a-2}x_4(g_3+g_4)g_6 
+ (x_0^{a-2}x_4+1)g_4g_6.
\eeq

(\ref{ex2-2-1})  
 From the relations   (\ref{ex2-minred-g4})  and (\ref{ex2-minred-g5})
we conclude that  $(g_4, g_5)^3 \subseteq JI^2$ and hence $JI^2 = I^3$. This implies that  $\aa (I) \leq 4$. 
Since the $\height(I)=3$,  we have $\aa (I) \geq 3$.  If  $\aa (I) =3$, then $\height(I) = \aa(I)$ and hence by \cite[Theorem~1.3]{huneke-rees},   $I$ is generated by a regular sequence which leads to a contradiction as $I$ is minimally generated by $6$ elements. Hence $\aa (I) = 4$ and $\rr(I) \leq 2$.

\noindent
(\ref{ex2-2-2}) 
If $\rr(I)=1$, then  $\mu(I^2) = \mu(JI) =18$ by \cite[Corollary~5.8]{cort-zar}. We claim  that $\mu(I^2) =  19$. 
Since $g_1, \ldots, g_6$ is a minimal generating set  for $I$, $\mu(I^2) \leq {7 \choose 2}=21$. The relations
 (\ref{ex2-minred-g4}) and (\ref{ex2-minred-g5}) imply that $\mu(I^2) \leq 19$. 
 To prove the claim it is enough to show that  $g_2g_6-g_3g_5 \not \in JI$. 
We  write 
\beq
\label{groebner-ji}
\q \p_a
&=& (g_1^2,
g_1g_2,
g_2^2 -g_1g_3,
g_2^2,
 g_2 g_3, 
 g_1 g_6-  g_2 g_5+  g_3( g_4+ g_3),
  g_1 g_4, 
 g_2 g_4, \\ \nonumber && 
g_3g_4, 
 g_1g_5 + g_2g_4, 
 g_2 g_5, 
 g_2 g_6, g_3 g_6
 g_3g_5-g_2g_6,
g_4^2,
g_4g_5 + g_3g_5-g_2g_6,\\ \nonumber && 
g_4 g_6,
g_5 g_6 - x_3^{a-1} g_2 g_4,
g_6^2 - x_3^{a-1} g_3  g_4).
\eeq
 Consider the polynomial ring $S = \kk[x_0, \ldots, x_4]$. We put the graded-rev-lex order on $S$ with $x_0 > x_1 > x_2 > x_3 >x_4$. Computing the S-polynomials of generators listed in  (\ref{groebner-ji}) we  conclude that it  is a Gr\"{o}bner basis of $\q \p_a$. 
 Let $LI( \q \p_a)$ denote the initial ideal  of $\q \p_a$. Then 
 \beqn
LI( \q \p_a) &=& (x_2^4,
x_1x_2^3,
x_0x_2^3,
x_1^2x_2^2,
x_1^3x_2,
x_1^4,
x_2^2x_3^{a+1},
x_1x_2x_3^{a+1},
x_1^2x_3^{a+1},
x_2^3x_3^{a},
x_1x_2^2x_3^{a},\\&&
x_1^2x_2x_3^{a},
x_1^3x_3^{a},
x_0^2x_2^2x_3^{a},
x_3^{2a+2},
x_2x_3^{2a+1},
x_1x_3^{2a+1},
x_0x_3^{2a+1},
x_0x_2x_3^{2a}).
 \eeqn
 Since 
 \beqn
 g_2 g_6 - g_3 g_5 = x_0 x_2^2 x_3^a
-x_0 x_1 x_3^{a+1}+x_0^{a-1} x_1^3 x_4-2 x_0^a x_1 x_2 x_4+x_0^{a+1} x_3 x_4,
\eeqn
the initial term of $g_2 g_6 - g_3 g_5$, is $x_0 x_2^2 x_3^a \not \in LI( \q \p_a)$.  Hence $g_2 g_6 -g_3g_5 \not \in \q \p_a$. 
We claim that $g_2 g_6 -g_3g_5  \not \in JI = \q \p_a S_{\m}$. Suppose $g_2 g_6 -g_3 g_5  \in J I$. Then there exists an element  $\alpha \in \m = (x_0, x_1, x_2, x_3, x_4)$ such that $(1-\alpha) (g_2 g_6 - g_3 g_5) \in \q \p_a$. 
Note that the initial term  of $x_i(  g_2 g_6 - g_3 g_5 )$ is $x_i(  x_0 x_2^2 x_3^a ) \in  LI( \q \p_a)$ for $i=0,1,2,3$.  Hence,  we only need to consider $\alpha = \alpha_1 x_4 + \cdots +  \alpha_b x_4^b$ where  $b >0$ and $\alpha_i \in k$. 
Then the initial term of 
$(1  - \alpha)  (g_2 g_6 - g_3 g_5) $  is $a_b x_0x_2^2x_3^a x_4^b\not \in LI( \q \p_a)$. 
 But $(1-\alpha) (g_2 g_6 - g_3 g_5) \in \q \p_a $  implies that $- \alpha (g_2 g_6 - g_3 g_5) = (1-\alpha) (g_2 g_6 - g_3 g_5) - (g_2 g_6 - g_3 g_5)  \in \q \p_a$  which leads to a contradiction. Hence 
$g_2 g_6 - g_3 g_5  \not \in J I$.  This implies that  $\mu(I^2) \geq 19$  and hence $\rr(I)=2$.
 
\noindent
(\ref{ex2-2-3}) 
Since $J$ is a minimal reduction of $I$ and   $\rr(I) = 2$,
 \beq
 \label{ex2-length-fc}
 \frac{\Fi_I}{J \Fi_I} 
 = \frac{R}{\m} \oplus \frac{I }{J + \m I} \oplus \frac{I^2}{JI + \m I^2}. 
  \eeq
Since  $J$ is a minimal reduction of $I$, $J \cap \m I = \m J$ and we have the exact sequence
\beq
 \label{ex2-nog}
                  0 
\longrightarrow \frac{J + \m I }{\m I } \cong \frac{J}{\m J} 
\longrightarrow  \frac{I}{\m I} 
\longrightarrow \frac{I}{J + \m I}
\longrightarrow 0.
\eeq
From (\ref{ex2-length-fc}) and (\ref{ex2-nog}) we get
  \beq
   \label{comp-length}
  \ell \left(  \frac{\Fi_I}{J \Fi_I} \right)
  &=& \ell \left(   \frac{R}{\m} \right)
  +  \ell \left(   \frac{I}{ J + \m I} \right) 
  + \ell \left(  \frac{I^2}{JI + \m I^2}\right)\\ \nonumber
  &= &\ell \left(   \frac{R}{\m} \right)
  +  \ell \left(   \frac{I}{\m I} \right) -  \ell \left(   \frac{J}{\m J} \right)
  + \ell \left(  \frac{I^2}{JI + \m I^2}\right)\\ \nonumber
   &=& 1 + (6-4) + 1\\ \nonumber
   &=&4.
  \eeq
  Let $\Fi_I = R/ \m[u_1, \ldots, u_6]/K $.
From the relations (\ref{ex2-minred-g4}) and (\ref{ex2-minred-g5})  we get that $u_5^2 - u_1u_4 + u_2 u_6 - u_3u_5,   u_6^2   - u_5u_4 \in K$.
Thus from  (\ref{comp-length}) 
we get 
  \beq
  \label{seq-ineq}
  4= 
\ell \left(  \frac{\Fi_I}{J \Fi_I} \right)
 =\ell \left(  \frac{R/ \m R[u_1, \ldots, u_6]}{K  + (u_1, u_2, \ldots, u_4)} \right)
\leq \ell \left(  \frac{R/ \m R[u_1, \ldots, u_6]}{ (u_5^2, u_6^2)  + (u_1, u_2, \ldots, u_4)} \right)
=4.
  \eeq
  Hence equality holds in (\ref{seq-ineq}) which implies that $K = (u_5^2 - u_1u_4 + u_2 u_6 - u_3u_5,   u_6^2   - u_5u_4 )$
 and the Hilbert Series of $\Fi_I$ is 
\beqn
H(\Fi_I, u) = \frac{(1-u^2)^2}{(1-u)^6} = \frac{1 + 2u+u^2}{(1-u)^4}. 
\eeqn 
We compare  the terms of degree $t$  and  using the fact that $\p_a$ is homogenous ideal, 
\beqn
\mu(\p_a^t) = \mu(I^t) =  {t+3 \choose 3} + 2 {t+2 \choose 3} + {t+1 \choose 3}.
\eeqn

(\ref{ex2-4})
For all $i \geq 0$, the $i$-th a-invariant of the Rees algebra is   ${a_i} (\R_I) = \max \{ n |  [H^i_{  (\R_I)_{+} }  (\R_I)]_n \not = 0\} $ and $\reg(\R_I) = \max \{ a_i + i | i \geq 0\}$. Put  $J^{[k]} = (g_1^k, g_2^k, (g_3 + g_4)^k, g_6^k)$.   Then
\beqn
 [H^4_{  (\R_I)_{+} }  (\R_I)]_n= \limnk \frac{I^{4k + n}}{J^{[k]} I^{3k +n}}.
\eeqn
As $JI^2 = I^3$,  for  all $n \geq -1$, for all $k \geq 1$,
\beqn
J^{[k]} I^{3k  + n} = J^{[k]} J^{3(k-1)} I^{n+3} = J^{ k + 3(k-1)} I^{n+3} = J^{4k-3} I^{n+3} = I^{4k+n}
\eeqn
which implies that $ [H^4_{  (\R_I)_{+} }  (\R_I)]_{n}=0$ for all $n \geq -1$.  We claim that   
\beqn
 [H^4_{  (\R_I)_{+} }  (\R_I)]_{-2}
=\frac{I^{4k-2}}{ J^{[k]} I^{3k -2} }
= \left( \frac{\Fi_I  }{J^{[k]} \Fi_I } \right)_{4k-2}  \not =0.
\eeqn
Since $\Fi_I$ is Cohen-Macaulay,  $g_1^k, g_2^k, (g_3+g_4)^k, g_6^k$ is a regular sequence  in $\Fi_I$ and hence the Hilbert series of 
$\Fi_I / J^{[k]} \Fi_I$ is 
\beq
\label{ex2-hilbert}
H \left( \frac{\Fi_I  }{J^{[k]} \Fi_I },  u \right)
= H \left( \Fi_I \right) (1-u^k)^4
= (1 + 2u + u^2)(1+u + \cdots + u^{k-1})^4.
\eeq
From the Hilbert series in (\ref{ex2-hilbert}) it follows that for all $k \geq 1$, 
\beqn
\dim \left( \frac{\Fi_I  }{J^{[k]} \Fi_I } \right)_{4k-2} = \dim \left( \frac{I^{4k-2}}{ J^{[k]} I^{3k -2} }\right) = 1, 
\eeqn
which  implies that $[H^4_{  (\R_I)_{+} }  (\R_I)]_{4k-2} =1$. Hence $a_4(\R_I) = -2$ which gives $\reg( \R_I) \geq 2$. 
\end{proof}

\begin{exmp}
\label{ex3}
Let $a \geq 3$ and let ${\p}_a={\mathcal I}(1,2,3,3a+2) 
\subseteq R := k[x_0, \ldots, x_4]$ be the defining ideal of the monomial curve  ${\mathcal C}(1,2,3,3a+2)$. 
 Then
 \been
 \item
 \label{ex3-1}
 $\p_a = (g_1, g_2, g_3, g_4, g_5)$ where 
$g_1=x_2 x_3^a  - x_0^a x_4 + x_2^2  - x_1 x_3$,
$g_2=x_1 x_2-x_0 x_3$,
$g_3=x_1^2-x_0 x_2$, 
$g_4=x_3^{a+1}-x_0^{a-1} x_1x_4$ and 
$g_5=x_2 x_3^a-x_0^a x_4$. 
 \item
 \label{ex3-2}
 Put $I = (\p_a)_{\m}$.
 
 \been
 \item
  \label{ex3-2-1}
  Then  $\aa(I) = 4$ and $\rr(I)=2$.
 
 \item
  \label{ex3-3-2-2}
  For all $t \geq 0$, $\mu(\p_a^t) = {t+3 \choose 3} +  {t+2 \choose 3} + {t+1 \choose 3} $.

 \item
   \label{ex3-3-1}
 $\reg(\Fi_I)= 2$.
 \eeen
  \item
  \label{ex3-4} $\reg ( \R_I) \geq 2$. 
\eeen
\end{exmp}
\begin{proof}
(\ref{ex3-1}) 
Put  $I_a = (g_1, g_2, g_3, g_4, g_5)$. One can verify that 
$I_a \subseteq \p_a$. 
Note that  $x_0, x_4$ is a system of parameters for $R/ I_a$ and 
$I_a + (x_0, x_4) 
= (x_0,  x_2^2  - x_1 x_3, x_1 x_2, x_1^2,  x_3^{a+1}, x_2 x_3^a,  x_4)$. 
Putting the graded-rev-lex order we get that the leading ideal of $I_a + (x_0, x_4)$ is 
$(x_0, x_1^2, x_1x_2, x_2^2,   x_2 x_3^a,x_3^{a+1}, x_4)$. 
Moreover, $(x_1^2, x_1x_2, x_2^2, x_2x_3^a, x_3^{a+1}) 
= (x_1, x_2^2, x_3^a) \cap (x_1^2, x_2, x_3^{a+1})$. Hence from the exact sequence 
\beqn
   0 
   \longrightarrow \frac{k[x_1, x_2, x_3]}{ (x_1^2, x_1x_2, x_2^2, x_2x_3^a, x_3^{a+1})}
   \longrightarrow \frac{k[x_1, x_2, x_3]}{ (x_1,  x_2^2, x_3^a) }
            \oplus \frac{k[x_1, x_2, x_3]}{ (x_1^2, x_2,  x_3^{a+1})}
   \longrightarrow \frac{k[x_1, x_2, x_3]}{ (x_1, x_2, x_3^a)}
   \longrightarrow 0
   \eeqn
   we get 
   \beq
   \label{length 3a+2} 
   &&\ell\left( \frac{k[x_1, x_2, x_3]}{ (x_1^2, x_1x_2, x_2^2, x_2x_3^a, x_3^{a+1})}\right)\\  \nonumber
   &=& \ell\left( \frac{k[x_1, x_2, x_3]}{ (x_1,  x_2^2, x_3^a)} \right)
   + \ell\left(  \frac{k[x_1, x_2, x_3]}{ (x_1^2,  x_2, x_3^{a+1})} \right)
   - \ell \left( \frac{k[x_1, x_2, x_3]}{ (x_1, x_2, x_3^a)}\right)\\ \nonumber
   &=& 2a + 2(a+1) -a = 3a+2. 
   \eeq
   Put  $\m= (x_0, x_1, x_2, x_3, x_4)$. Then 
\begin{align}
\label{multiplicity 3a+2} 
  3a +2
& =  e( (x_0, x_4) ; R/ \p_a) \\  \nonumber
&=  e( (x_0, x_4) ; R_{\m}/ \p_a R_{\m})  & \mbox{\cite[Theorem~15, page 333]{northcott-lessons}}\\ \nonumber
&\leq   \ell \left(  \frac{R_{\m} }{ (\p_a + (x_0, x_4)) R_{\m}} \right) & \\ \nonumber
&\leq   \ell \left(  \frac{R}{ I_a + (x_0, x_4)} \right) & \\ \nonumber
&= \ell \left(  \frac{R}{ (x_0, x_1^2, x_1x_2, x_2^2, x_2x_3^a,x_3^{a+1},   x_4)} \right) & \\ \nonumber
&= \ell \left(  \frac{k[x_1, x_2, x_3]}{ (x_1^2, x_1x_2, x_2^2, x_2x_3^a, x_3^{a+1})} \right) & \\ \nonumber
&= 3a +2& \mbox{[from (\ref{length 3a+2})]} .
\end{align}
Hence $I_a  + (x_0, x_4) = \p_a + (x_0, x_4)$.
The rest of the proof is similar to the proof in Example~\ref{ex1}(\ref{ex1-1}).
This proves (\ref{ex3-1}).

(\ref{ex3-2-1})  Let $I= (\p_a)_{\m}$.  We claim that $J= (g_1, g_2, g_3, g_4)$ is a minimal reduction of $I$.
The relation
\beq
\label{ex3-redno} \nonumber
0 =  x_0^{a-2} x_3^{a-1} x_4 g_2^3 -x_0^{a-3}x_1x_3^{a-1}x_4 g_2^2 g_3 +x_0^{a-3} x_2 x_3^{a-1} x_4 g_2 g_3^2
+x_3^{a-1} g_1^2 g_4  \\              
-x_0^{a-2} x_4 g_3^2 g_5                              -2 x_3^{a-1} g_1 g_4 g_5
+x_3^{a-1} g_4 g_5^2                     -g_3 g_4^2                                                +g_2 g_4 g_5-g_1 g_5^2+g_5^3
\eeq
implies that $JI^2 = I^3$ and $\aa(I)\leq 4$.  Since $\height(I)=3$, we have $\aa(I) \geq 3$.  If $\aa(I)=3$, then since $\height(I) = 3$ by \cite[Theorem~1.3]{huneke-rees},  $I$ is generated by a regular sequence which leads to a contradiction.  Hence $a(I) =4$ and that $J$ is a minimal reduction of $I$. 
From (\ref{ex3-redno}) we get that $\rr(I) \leq 2$.

Since $J$ is a minimal reduction of $I$, 
\beqn
\Fi_J &\cong&  R_{\m} / \m R_{\m}[u_1, u_2, u_3, u_4]\\
\Fi_I &\cong& \frac{R_{\m} / \m R_{\m}[u_1, u_2, u_3, u_4, u_5]}{ ( f)}
\eeqn
where $f= u_5^3 -u_1 u_5^2 +u_2 u_4 u_5-u_3 u_4^2$. 
Therefore  
  the Hilbert series of $\Fi_I$ is 
\beq
\label{ex-3-hilbert series}
     H(\Fi_I, u) 
= \frac{1 +  u + u^2}{(1-u)^4}
\eeq
and we conclude that $\rr(I)=2$.

  (\ref{ex3-3-2-2})
As $\p_a$ is homogenous ideal, from (\ref{ex-3-hilbert series}), for all $t \geq 1$
\beqn
\mu(\p_a^t) = \mu(I^t)  = {t+3 \choose 3} +  {t+2 \choose 3} + {t+1 \choose 3}.
\eeqn

(\ref{ex3-3-1})
 Put $K = R_{\m} / \m R_{\m}$.
The exact sequence 
\beqn
      0 
\lto K [u_1, u_2, u_3, u_4, u_5] [-3]
\lto K[u_1, u_2, u_3, u_4, u_5] 
\lto \frac{K[u_1, u_2, u_3, u_4, u_5]} { (f) }
\lto 0 
\eeqn
 imples that
 $\reg(\Fi_I)= 2$. 

(\ref{ex3-4})
For all $i \geq 0$, the $i$-th a-invariant of the Rees algebra of $I$ is   ${a_i} (\R_I) := \max \{ n |  [H^i_{  (\R_I)_{+} }  (\R_I)]_n \not = 0\} $ and $\reg(\R_I) = \max \{ a_i + i | i \geq 0\}$. Put  $J^{[k]} = (g_1^k, g_2^k, g_3^k, g_4^k)$. Then
\beq
\label{ex-3 cohomology}
 [H^4_{  (\R_I)_{+} }  (\R_I)]_n= \limnk \frac{I^{4k + n}}{J^{[k]} I^{3k +n}}.
\eeq
As $JI^2 = I^3$,  for  all $n \geq -1$, for all $k \geq 1$,
\beqn
J^{[k]} I^{3k  + n} = J^{[k]} J^{3(k-1)} I^{n+3} = J^{ k + 3(k-1)} I^{n+3} = J^{4k-3} I^{n+3} = I^{4k+n}
\eeqn
which implies that $ [H^4_{  (\R_I)_{+} }  (\R_I)]_{n}=0$ for all $n \geq -1$.  We claim that   $ [H^4_{  (\R_I)_{+} }  (\R_I)]_{-2} \not=0$.

As $\Fi_I$ is Cohen-Macaulay and $g_1^k, g_2^k, g_3^k, g_4^k$ is a system of parameters and hence a regular sequence in $\Fi_I$. Therefore,  the Hilbert Series of 
$\Fi_I / J^{[k]} \Fi_I$ is 
\beqn
H \left( \frac{\Fi_I  }{J^{[k]} \Fi_I },  u \right)
= H \left( \Fi_I \right) (1-u^k)^4
= (1 + u + u^2)(1+u + \cdots + u^{k-1})^4
\eeqn 
which implies that 
\beq
\label{ex-2 fibercone}
\dim \left( \frac{\Fi_I  }{J^{[k]} \Fi_I } \right)_{4k-2} = \dim \left( \frac{I^{4k-2}}{ J^{[k]} I^{3k -2} }\right) = 1.
\eeq
From (\ref{ex-3 cohomology}) and  (\ref{ex-2 fibercone}) we get $\ell \left(  [H^4_{  (\R_I)_{+} }  (\R_I)]_{4k-2}\right)=1 >0$. 
Hence we conclude that  $a_4(\R_I) = -2$ which implies that  $\reg( \R_I) \geq 2$. 
\end{proof}

\end{document}